\documentclass[12pt, reqno]{amsart}
\usepackage{amsmath, amsthm, amscd, amsfonts, amssymb, graphicx, color}
\usepackage{amssymb,amsmath,amsfonts,latexsym}
\usepackage{bm}
\usepackage[all,cmtip]{xy}
\usepackage{amscd}

\newtheorem{theorem}{Theorem}[section]

\newtheorem{proposition}[theorem]{Proposition}
\newtheorem{corollary}[theorem]{Corollary}
\theoremstyle{definition}

\newtheorem{example}[theorem]{Example}

\theoremstyle{remark}
\newtheorem{remark}[theorem]{Remark}
\numberwithin{equation}{section}

\newcommand{\D}{\mathbb{D}}

\begin{document}



\title[Dynamics]{{Dynamics of weighted backward shifts on certain analytic function spaces}}

\author[Das]{Bibhash Kumar Das}
\address{Indian Institute of Technology Bhubaneswar, Jatni Rd, Khordha - 752050, India}
\email{bkd11@iitbbs.ac.in}

\author[Mundayadan]{Aneesh Mundayadan}
\address{Indian Institute of Technology Bhubaneswar, Jatni Rd, Khordha - 752050, India}
\email{aneesh@iitbbs.ac.in}

\subjclass[2010]{Primary 47A16, 46E22, 32K05, 47B32; Secondary
47B37, 37A99.}

\keywords{hypercyclic operator, chaos, mixing operator, weighted shift operator, reproducing kernel Hilbert
space}


\begin{abstract}
We introduce the Banach spaces $\ell^p_{a,b}$ and $c_{0,a,b}$, of analytic functions on the unit disc, having normalized Schauder bases consisting of polynomials of the  form $f_n(z)=(a_n+b_nz)z^n, ~~n\geq0$, where $\{f_n\}$ is assumed to be equivalent to the standard basis in $\ell^p$ and $c_0$, respectively. We study the weighted backward shift operator $B_w$ on these spaces, and obtain necessary and sufficient conditions for $B_w$ to be bounded, and prove that, under some mild assumptions on $\{a_n\}$ and $\{b_n\}$, the operator $B_w$ is similar to a compact perturbation of a weighted backward shift on the sequence spaces $\ell^p$ or $c_0$. Further, we study the hypercyclicity, mixing, and chaos of $B_w$, and establish the existence of hypercyclic subspaces for $B_w$ by computing its essential spectrum. Similar results are obtained for a function of $B_w$ on $\ell^p_{a,b}$ and $c_{0,a,b}$.
\end{abstract}
\maketitle

\noindent
\tableofcontents
\section{Introduction}

For a suitable sequence $\nu:=(\nu_n)_{n=0}^{\infty}$ of complex numbers, the weighted spaces $\ell^p_{\nu}$ and $c_{0,\nu}$ are defined to be the spaces of all analytic functions of the form $f(z)=\sum_{n=0}^{\infty} \lambda_nz^n$ such that $\sum_n \nu_n|\lambda_n|^p<\infty$ and $\lim_n |\nu_n\lambda_n|=0$, respectively. The weighted backward shift operator $B_w$ on such a space is given by
\begin{center}
 $B_w\big(\sum_{n=0}^{\infty}\lambda_nz^n\big)=\sum_{n=0}^{\infty}w_{n+1}\lambda_{n+1}z^n.$
\end{center}
The spaces $\ell^p_{\nu}$ and $c_{0,\nu}$, and the operator $B_w$ play central roles in operator theory, particularly in linear dynamics. Motivated by this, we introduce $\ell^p_{a,b}$ and $c_{0,a,b}$, which are Banach spaces of analytic functions on the unit disc in the complex plane, having normalized Schauder bases of the form 
\[
\big\{(a_n+b_nz)z^{n}:n\geq 0\big\},
\]
equivalent to the standard bases in $\ell^p(\mathbb{N})$ and $c_0(\mathbb{N})$, respectively. We primarily study the dynamical properties (hypercyclicity, mixing, periodic vectors, and chaos) of $B_w$ on these spaces. The spaces $\ell^p_{a,b}$ and $c_{0,a,b}$ are of independent and general interest as well because they are the ``next best" Banach sequence spaces, compared to the weighted $\ell^p$ and $c_0$ spaces.  For a general theory of classical weighted shifts, we refer to Shields \cite{Shields}. In linear dynamics, shifts received a major attention through Godefroy and Shapiro \cite{Godefroy-Shapiro}, Kitai \cite{Kitai} and Salas \cite{Salas}. For an account on the fundamentals of linear dynamics, see the monographs by Bayart and Matheron \cite{Bayart-Matheron} and Grosse-Erdmann and Peris \cite{Erdmann-Peris}.

 An operator $T$ on a separable Banach space $X$ is said to be \textit{hypercyclic} if there exists $x\in X$, known as a \textit{hypercyclic vector} for $T$, such that the orbit$\{x,Tx,T^2x,\cdots\}$ is dense in $X$. If a hypercyclic operator $T$ on $X$ has a dense set of periodic vectors, then $T$ is called \textit{chaotic}. Recall that a vector $y\in X$ is periodic for $T$ if its orbit under $T$ is periodic, that is, $T^py=y$ for some $p$. An operator $T$ on $X$ is said to be \textit{topologically transitive} if, for two non-empty open sets $U_1$ and $U_2$ of $X$, there exists a natural number $k$ such that $T^k(U_1)\cap U_2\neq \phi$. The transitivity notion is equivalent to that of hypercyclicity, assuming the separability of the underlying Banach space $X$. A strong form of transitivity is the topological mixing: an operator $T$ is \textit{topologically mixing} on $X$ if, for any two non-empty open sets $U_1$ and $U_2$ of $X$, there exists $N$, a natural number, such that $T^n(U_1)\cap U_2\neq \phi$ for all $n\geq N$. Mixing and chaos are stronger than the hypercyclicity; however, they are not comparable in general. Several familiar operators including weighted shifts on sequence spaces, and composition operators and differential operators on analytic function spaces exhibit the hypercyclic, mixing and chaotic properties. The study is intimately related to classical areas such as complex function theory, dynamical systems, and operator theory, cf. \cite{Bayart-Matheron} and \cite{Erdmann-Peris}.

We will make use of the following standard criteria in linear dynamics for establishing the hypercyclic and chaotic properties of the backward shift. Different versions of these criteria are available in the literature, cf. \cite{Bayart-Matheron} and \cite{Erdmann-Peris}.

\begin{theorem}\label{thm-hypc} (\textsf{Gethner-Shapiro Criterion} \cite{Gethner-Shapiro})
Let $T$ be a bounded operator on a separable Banach space $X$, and let $X_0$ be a dense subset of $X$. If $\{n_k\} \subseteq \mathbb{N}$ is a strictly
increasing sequence and $S:X_0\mapsto X_0$ is a map such that, for each $x\in X_0$,
\[
\lim_{k\rightarrow \infty} T^{n_k}x=0=\lim_{k\rightarrow \infty}S^{n_k}x,
\]
and
\[
TSx=x,
\]
then $T$ is hypercyclic. Moreover, if $n_k=k$ for all $k\geq 1$, then $T$ is mixing on $X$.
\end{theorem}

A similar criterion, known as the chaoticity criterion, has been used to obtain chaotic operators in Banach spaces, cf. \cite{Bonilla-Erdmann1}. This criterion is very strong, and it has several powerful implications in linear dynamics; see \cite{Bayart-Matheron} and \cite{Erdmann-Peris}.

\begin{theorem}(\textsf{Chaoticity Criterion} \cite{Bonilla-Erdmann1}) \label{chaos}
Let $X$ be a separable Banach space, $X_0$ be a dense set in $X$, and let $T$ be a bounded operator on $X$. If there exists a map $S:X_0\rightarrow X_0$ such that
\[
\displaystyle \sum_{n\geq 0} T^nx\quad \mbox{and}\quad \sum_{n\geq 0} S^nx,
\]
are unconditionally convergent, and
\[
TSx=x
\]
for each $x\in X_0$, then the operator $T$ is chaotic and mixing on $X$. 
\end{theorem}

\begin{theorem} (\textsf{Eigenvalue Criteria} \cite{Bayart-Matheron} and \cite{Erdmann-Peris}) \label{aa}
    	Let $T$ be a bounded operator on a separable complex Banach space $X$. Suppose that the subspaces
		\[
  X_{0}=span \{x \in X;\,\,\, Tx=\lambda x ,\,\,\,\lambda \in \mathbb{C},\,\,\, |\lambda|<1\}, 
  \]
  \[
  X_{1}=span\{x \in X;\,\,\, Tx=\lambda x ,\,\,\, \lambda \in \mathbb{C} ,\,\,\, |\lambda|> 1\},
  \]
  are dense in X.Then T is mixing, and in particular hypercyclic. In addition, if the subspace
		\[
     X_{2}= span \{x \in X;\,\,\, Tx=e^{\alpha \pi i}x ,\,\,\, \alpha \in \mathbb{Q}\},
         \] 
		is dense in X, then T is chaotic.
\end{theorem}

Rolewicz \cite{Rolewicz} showed that $\lambda B$ is hypercyclic on the sequence space $\ell^p$, where $|\lambda|>1$, $B$ is the unweighted backward shift and $1\leq p<\infty$. Salas \cite{Salas} provided a complete characterization of the hypercyclicity of classical unilateral and bilateral shifts. Hypercyclicity and chaos of weighted shifts on $F$-sequence spaces were characterized by Grosse-Erdmann \cite{Erdmann}. Also, see Costakis and Sambarino \cite{Costakis} for mixing shifts, and Bonet, Kalmes and Peris \cite{Bonet2} for dynamics of shifts on non-metrizable sequence spaces. It is well known that the backward shift on the Bergman space of the unit disc is a mixing and non-chaotic operator, cf. Gethner and Shapiro \cite{Gethner-Shapiro} and Grosse-Erdmann \cite{Erdmann}, respectively. We also refer to Bonet \cite{Bonet1}, Beise and M\"{u}ller \cite{Beise-Muller1}, Beise, Meyrath and M\"{u}ller \cite{Beise-Meyrath-Muller}, Bourdon and Shapiro \cite{Bourdon-Shapiro}, and M\"{u}ller and Maike \cite{Muller-Maike} for the dynamics related to the backward shift on analytic function spaces, (Bergman spaces, mostly).
 
  The paper is organized as follows. In Section $2$, we introduce the analytic function spaces $\ell^p_{a,b}$ and $c_{0,a,b}$, and obtain necessary and sufficient conditions for the weighted shift $B_w$ to be bounded on these spaces. In Section $3$, under some mild conditions we show that the operator $B_w$ on $\ell^p_{a,b}$ and $c_{0,a,b}$ is similar to a compact perturbation of a classical weighted shift on $\ell^p$ and $c_0$, respectively. Using this result, we compute the essential spectrum of $B_w$, which establishes the existence of hypercyclic subspaces for $B_w$. In Section $4$, we characterize the hypercyclicity, mixing, and chaos of $B_w$ in $\ell^p_{a,b}$ and $c_{0,a,b}$. Further, the dynamical properties of functions of the weighted shift $B_w$ are also studied.

\section{The analytic function spaces $\ell^p_{a,b}$ and $c_{0,a,b}$, and the co-ordinate functionals }

In this section, we introduce the spaces $\ell^p_{a,b}$, $1\leq p<\infty$ and $c_{0,a,b}$, and primarily, show that the co-ordinate functionals are bounded, and estimate useful upper bounds for their norms. We need the concept of equivalent Schauder bases in Banach spaces. Let $X$ and $Y$ be Banach spaces. Two Schauder bases, $\{u_{n}\}_{n=0}^{\infty}$ of X and $\{v_{n}\}_{n=0}^{\infty}$ of Y, are \textit{equivalent} if the convergence of $\sum_{n=1}^{\infty}a_{n}u_{n}$ is equivalent to that of $\sum_{n=1}^{\infty}a_{n}v_{n}$; see, for instance \cite{Lindenstrauss}, page $5$. 

Let $a=\{a_n\}_{n=0}^{\infty}$ and $b=\{b_n\}_{n=0}^{\infty}$ be two sequences of complex sequences, where $a_n\neq 0$ for all $n$. Let $\ell^p_{a,b}$ denote the space of all analytic functions on the unit disc, having a normalized Schauder basis consisting of polynomials of the  form $f_n(z)=(a_n+b_nz)z^n, ~~n\geq0$, which is equivalent to the standard basis in $\ell^p$, $1\leq p < \infty$, that is,  $f(z)_{}=\sum_{n=0}^{\infty}\lambda_{n}f_{n}(z) \in \ell^p_{a,b},$ if and only if
\[
\lVert f \rVert_{\ell^p_{a,b}}=\left(\sum_{n=0}^{\infty}\lvert\lambda_{n}\rvert^{p}\right)^{\frac{1}{p}}.
\]
For $\ell^2_{a,b}$, we assume something more, that is, $\{f_n\}$ to be an orthonormal basis.

Note that, for $a_{n}=1$ and $b_{n}=0$ for all $n\geq 0,$ we get the standard $\ell^{p},$ $1\leq p<\infty$ space of all analytic functions on the unit disc. We can similarly define the space $c_{0,a,b}$ which consists of analytic functions, for which $\{(a_n+b_nz)z^n:n\geq 0\}$ forms a normalized Schauder basis, equivalent to the standard basis in $c_0$. Moreover, for $f(z)=\sum_n \lambda_n f_n(z)$ in $c_{0,a,b}$, we have $\lim_{n\rightarrow \infty}\lambda_n=0$. Define the norm in $c_{0,a,b}$ as
\[
\|f\|_{c_{0,a,b}}:=\sup_{n\geq 0} |\lambda_n|.
\]
The Hilbert space case of $\ell^2_{a,b}$, was studied by Adams and McGuire, cf. \cite{Adams-McGuire} in the context of tridiagonal kernels and commutants of the forward shift operator. For $\ell^p_{a,b}$, observe that $\{f_n\}$ is an unconditional basis. When $b_n=0$ for all $n$, then $\ell^p_{a,b}$ and $c_{0,a,b}$ become the standard weighted $\ell^p$ and $c_0$ spaces. Below, we show that there is a natural open disc on which every function in $\ell^p_{a,b}$ is analytic. The same also holds for $c_{0,a,b}$. In the proof, we will also see that the norm expansion $f=\sum_{n\geq 0} \lambda_n f_n=\sum_{n\geq 0} \lambda_n (a_nz^n+b_nz^{n+1})$ in $\ell^p_{a,b}$ and $c_{0,a,b}$ imply that the series can be rearranged so that the Taylor series is absolutely convergent to $f(z)$ on the respective domains $\mathcal{D}_1$, $\mathcal{D}_2$ or $\mathcal{D}_3$.

\begin{theorem}
    For given complex sequences, $a=\{a_n\}$ and $b=\{b_n\}$, set
    \[
    \mathcal{D}_{1}:=\left\{ z \in \mathbb{C}:\sum_{n=0}^{\infty} \left(( \lvert a_{n}\rvert+\lvert b_{n}\rvert)\lvert z \rvert^{n} \right)^{q}<\infty\right\},
    \]
    where $\frac{1}{p}+\frac{1}{q}=1,$
    \[
    \mathcal{D}_{2}:==\left\{ z \in \mathbb{C}:\sup_{n\geq 0} ~(\lvert a_{n}\rvert+\lvert b_{n}\rvert)\lvert z\rvert^{n}<\infty\right\},
    \]
    and 
   \[
    \mathcal{D}_{3}:==\left\{ z \in \mathbb{C}:\sum_{n=0}^{\infty} (\lvert a_{n}\rvert+\lvert b_{n}\rvert)\lvert z \rvert^{n} <\infty\right\}.
    \] 
    If $f=\sum_{n=0}^{\infty}\lambda_{n}f_{n}\in \ell^{p}_{a,b},$ $1<p<\infty,$ then the series $\sum_{n=0}^{\infty}\lambda_{n}(a_{n}+b_{n}z)z^{n} $ converges uniformly and absolutely on compact subsets of $\mathcal{D}_{1}$. Also, the evaluation functional $f\longmapsto f(\lambda),$ is bounded for each $\lambda \in \mathcal{D}_{1}$.\\ The same conclusions hold true with respect to the domains $\mathcal{D}_2$ and $\mathcal{D}_3$, respectively, for the spaces $\ell^1_{a,b}$ and $c_{0,a,b}$.
\end{theorem}
\begin{proof} We provide the proof for $\ell^p_{a,b}$ when $1<p<\infty$, and it is essentially similar for $\ell^1_{a,b}$ and $c_{0,a,b}$.

Let $f\in \ell^p_{a,b}$ for $1<p<\infty$. Then $f$ has a norm expansion $f=\sum_{n=0}^{\infty}\lambda_{n}f_{n}$ and $\|f\|_{\ell^p_{a,b}}^p=\sum_{n\geq 0} |\lambda_n|^p$. Fix a closed ball $K:=|z|\leq R$ that lies in $\mathcal{D}_1$. Note that, we have
    \[
    \sup_{z\in K} ~\sum_{n=0}^{\infty} \left(( \lvert a_{n}\rvert+\lvert b_{n}\rvert)\lvert z\rvert^{n} \right)^{q}<\infty,
    \]
    where $\frac{1}{p}+\frac{1}{q}=1.$ Now, by using the Hölder's inequality, we get
    \begin{eqnarray*}
    \sum_{n=0}^{\infty} \lvert \lambda_{n}(a_{n}+b_{n}z)z^{n}\rvert &\leq& \text{max} \{1, \lvert z\rvert\}\sum_{n=0}^{\infty}\lvert\lambda_{n}\rvert(\lvert a_{n}\rvert+\lvert b_{n}\rvert)\lvert z \rvert^{n}\\
    &\leq& \text{max} \{1, \lvert z\rvert\} \left(\sum_{n=0}^{\infty} \lvert\lambda_{n}\rvert^{p}\right)^{\frac{1}{p}} \left(\sum_{n=0}^{\infty} ((\lvert a_{n}\rvert+\lvert b_{n}\rvert)\lvert z \rvert^{n})^{q}\right)^{\frac{1}{q}}\\
    &=& C \lVert f \rVert_{\ell^{p}_{a,b}}<\infty,
     \end{eqnarray*}
     where
     \[
     C=\max_{z\in K} \left \{1,|z|, \left(\sum_{n=0}^{\infty} ((\lvert a_{n}\rvert+\lvert b_{n}\rvert)\lvert z \rvert^{n})^{q}\right)^{\frac{1}{q}}\right \}.
     \]
     Hence, the series $\sum_{n=0}^{\infty}\lambda_{n}(a_{n}+b_{n}z)z^{n} $ converges absolutely and uniformly on compact subsets of $\mathcal{D}_{1}.$
\end{proof} 

An important consequence of the above theorem is that on any closed disk in $\mathcal{D}_j$, $j=1, 2, 3$, the series $f=\sum_{n=0}^{\infty}\lambda_{n}f_{n}$ in $\ell^p_{a,b}$ and $c_{0,a,b}$ can be rearranged as a power series so that Taylor series of $f(z)$ about the origin, converges point-wise to $f(z)$, and hence, we have $f(z)=\lambda_{0}a_{0}+\sum_{n=1}^{\infty}(\lambda_{n}a_{n}+\lambda_{n-1}b_{n-1})z^{n}$ point-wise. Equivalently, the evaluation functionals are continuous.

\begin{center}
\textit{We will always assume that $\mathcal{D}_1=\mathcal{D}_2=\mathcal{D}_3=\mathbb{D}$, where $\mathbb{D}$ is the open unit disc in the complex plane.}
\end{center}

The case of $\ell^2_{a,b}$ is of special importance as it is a reproducing kernel Hilbert space. Occasionally, we will infer the dynamics of a weighted shift $B_w$ from the general properties of the kernels of the reproducing kernel spaces on which $B_w$ is defined. We briefly recall the basics and essential properties of analytic (scalar valued) reproducing kernel Hilbert spaces, and refer to Aronszajn \cite{Aro} and Paulsen and Raghupati \cite{Paulsen}. A function $k : \D \times \D \rightarrow \mathbb{C}$ is called an \textit{analytic kernel} (or a reproducing kernel) if $z\mapsto k(z,\zeta)$ is analytic for each fixed $\zeta \in \mathbb{D}$ and 
$$\sum_{i,j = 1}^n \overline{\alpha_i}\alpha_jk(\zeta_i, \zeta_j) \geq 0,$$
for all choices of $\zeta_1,\ldots,\zeta_n \in \D$, $\alpha_1,\ldots,\alpha_n\in \mathbb{C}$, and $n \in \mathbb{N}$. For an analytic kernel $k(z,\zeta)$ over $\mathbb{D}$, there exists a unique Hilbert space $\mathcal{H}(k)$ of analytic functions on $\D$ such that $\text{span}~\{k(\cdot, \zeta): \zeta \in \D\}$ is dense in $\mathcal{H}(k)$ and
\begin{equation}\label{RP}
f(\zeta)=\big\langle f,k(.,\zeta)\big\rangle_{\mathcal{H}(k)},
\end{equation}
for all $f\in \mathcal{H}(k)$, $\zeta \in \D$. Here, the symbol $k(\cdot, \zeta)$ denotes the function $z\mapsto k(z,\zeta)$ on $\mathbb{D}$. Moreover, for $\alpha_1,\ldots,\alpha_n\in \mathbb{C}$ and $\zeta_1,\ldots,\zeta_n\in \mathbb{D}$, 
$$\|\sum_{j=1}^n \alpha_jk(.,\zeta_j)\|_{\mathcal{H}(k)}^2=\sum_{i,j = 1}^n \overline{\alpha_i} \alpha_j k(\zeta_i, \zeta_j),$$
which can be seen from \eqref{RP}. The Hilbert space $\mathcal{H}(k)$ is called the \textit{analytic reproducing kernel Hilbert space} associated to the kernel $k(z,\zeta)$. From \eqref{RP} it follows that the evaluation functional $E_\zeta:\mathcal{H}(k)\rightarrow \mathbb{C}$ is bounded for all $\zeta\in \mathbb{D}$, where $E_\zeta(f)=f(\zeta), \hspace{.1cm} f\in \mathcal{H}(k)$. From \eqref{RP}, it also follows that $k(z,\zeta)$ is co-analytic in $\zeta$.  

The following are noteworthy: $\text{(1)}$ an analytic kernel $k(z,\zeta)$ is the kernel for a Hilbert space $\mathcal{H}$ of analytic functions on $\mathbb{D}$ if and only if all the evaluation functionals are bounded on $\mathcal{H}$, the span $\{k(.,\zeta):\zeta\in \mathbb{D}\}$ is dense in $\mathcal{H}$, and 
$f(\zeta)=\big\langle f,k(.,\zeta)\big\rangle_{\mathcal{H}}$ for all $f\in \mathcal{H}$ and $\zeta\in \mathbb{D}$. The kernel function has a formula, namely $k(z,\zeta)=\sum_{n\geq 0}e_n(z)\overline{e_n(\zeta)}$ for any orthonormal basis $\{e_n\}_{n\geq 0}$ of the space for which $k(z,\zeta)$ is the kernel, cf. \cite{Paulsen}.\\ $\text{(2)}$ Several well known spaces have reproducing kernels. A standard example for an analytic reproducing kernel space is the diagonal space $\mathcal{H}^2(\beta)$: for a given $\beta=\{\beta_n\}_{n=0}^{\infty}$ of strictly positive reals, this space consists of analytic functions $f(z)=\sum_{n\geq 0}\lambda_n z^n$ on $\mathbb{D}$ such that $\|f\|^2:=\sum_{n\geq 0} |\lambda_n|^2/{\beta_n}<\infty$. As $\sqrt{\beta_n}z^n$, $n\geq 0$, forms an orthonormal basis for $\mathcal{H}^2(\beta)$, its kernel is given by $\sum_{n\geq 0} \beta_nz^n\overline{\zeta}^n$. The classical spaces of Hardy, Bergman, and Dirichlet over the unit disc, are analytic reproducing kernel spaces, whose kernels can be explicitly written using the aforementioned formula for $k(z,\zeta)$.

Kernels are useful in linear dynamics as well. It is known that, if $k(z,\zeta)$ is an analytic scalar kernel, then the derivatives 
\[
\frac{\partial^nk(.,0)}{\partial \overline{\zeta}^n}
\]
can give information on the dynamics of the adjoint of the multiplication by the independent variable on $\mathcal{H}(k)$, see \cite{Mundayadan-Sarkar}. We will use the following facts on the co-ordinate functionals to derive the necessary parts in the characterization of hypercyclicity, mixing and chaos of $B_w$ on $\ell^p_{a,b}$ and $c_{0,a,b}$. This result is known for the case of reproducing kernel spaces, cf. \cite{Mundayadan-Sarkar}.

\begin{proposition}\label{evaluation} 
    Let $X$ be a Banach space of analytic functions on $\mathbb{D}$, having bounded evaluation functionals at every point in $\mathbb{D}$. Then
    \[
    \frac{d^{n}ev_{z}}{dz^{n}}|_{z=0} \in X^{*}
    \]
    with the action given by
    \[
    \frac{d^{n}ev_{z}}{dz^{n}}|_{z=0}(f)= f^{(n)}(0),
    \]
    for all $f \in X$ and $n\geq 0$.
\end{proposition}
\begin{proof}
     Let $ev_{z}$ be the evaluation functional at $z\in\mathbb{D}$, given by $ev_z(f)=f(z)$, where $f\in X$. Note that the $X^*$-valued function
     \[
     z \mapsto ev_{z}
     \]
     is point-wise analytic on $\mathbb{D}$, and hence it analytic in the norm of $X^*$ since the analyticity in the strong operator topology and the same in the operator norm topology are equivalent, cf. Chapter 5, Theorem 1.2, \cite{Taylor}. Thus, using the closed graph theorem, it follows that
 \[
 \frac{d^{n}ev_{z}}{dz^{n}}|_{z=0} \in X^{*}
 \]
 for $n\geq 0$. The power series expansion (in the norm of $X^*$)
 \[
 ev_{z}=\sum_{n\geq 0}\frac{1}{n!} z^{n}\frac{d^{n}ev_{z}}{dz^{n}}|_{z=0} ,
 \]
along with the observation
 \[
ev_{z}(f)=f(z)=\sum_{n\geq 0}\frac{1}{n!} z^{n} f^{
(n)}(0),
 \]
 for all $f \in X$ yields that
 \[
 \frac{d^{n}ev_{z}}{dz^{n}}|_{z=0}(f)= f^{(n)}(0),
 \]
 which completes the proof.
 \end{proof}

 The above result in the case of a Hilbert space of analytic functions can be stated in terms of the analytic kernel of the space, and its proof uses the analyticity of $k(z,\zeta)$ in the variable $z$, and co-analyticity in $\zeta$. Indeed, since the map $\zeta \mapsto k(.,\zeta)$ is co-analytic, we have the power series of $k(.,\zeta)$ in terms of $\overline{\zeta}$. Now, proceed as in the proof of the previous proposition. Also, see \cite{Mundayadan-Sarkar} and \cite{Curto}.

\begin{proposition}  \label{partial}
If $\mathcal{H}(k)$ is an analytic reproducing kernel space over $\mathbb{D}$, then 
\[
\frac{\partial^nk(.,0)}{\partial \overline{\zeta}^n} \in \mathcal{H}(k)\hspace{.5cm} \text{and}\hspace{.5cm} 
f^{(n)}(0)=\big\langle f, \frac{\partial^nk(.,0)}{\partial \overline{\zeta}^n}\big \rangle_{\mathcal{H}(k)},
\]
for all $n\geq 0$ and $f\in \mathcal{H}(k)$. Moreover,
\[
\left \| \frac{\partial^nk(.,0)}{\partial \overline{\zeta}^n}\right\|_{\mathcal{H}(k)}=\left(\frac{\partial^{2n}k}{\partial z^n\partial \overline{\zeta}^n} (0,0)\right)^{1/2}.
\]
\end{proposition}

We can now observe that the Hilbert space $\ell^2_{a,b}$ is a reproducing kernel Hilbert space, and $\{f_n\}_{n=0}^{\infty}$ forms an orthonormal basis, where 
\[
 f_{n}(z) =  (a_{n}+b_{n}z)z^{n}, ~~n\geq 0.
\]
Since, $k(z,\zeta)=\sum_{n \geq 0}  f_{n}(z)\overline{f_{n}(\zeta)},$ we get the kernel of $\ell^2_{a,b}$ as
\begin{equation}\label{a}
 k(z,\zeta)= |a_{0}|^{2}+ \sum_{n \geq 1} ( |a_{n}|^{2}+ |b_{n-1}|^{2} )z^{n}\overline{\zeta}^{n}+ \sum_{n \geq 0}  a_{n}\overline{b_{n}}z^{n}\overline{\zeta}^{n+1}+  \sum_{n \geq 0}  \overline{a_{n}}b_{n}z^{n+1}\overline{\zeta}^{n},
 \end{equation}
 for all $z,\zeta \in \mathbb{D}$. This is an example of an analytic tridiagonal kernel. The space $\ell^2_{a,b}$ is called a tridiagonal kernel space. For more on the terminology of tridiagonal kernels, we refer to Adams and McGuire \cite{Adams-McGuire} wherein the authors showed striking differences between the usual weighted forward shifts on $\ell^2$ and the tridiagonal shifts in terms of their commutants.
\begin{proposition}\label{N-L}
    Let $L_n$ denote the co-ordinate functional 
    \[
    f\mapsto \frac{f^{(n)}(0)}{n!},
    \]
    defined on $\ell^p_{a,b}$ or $c_{0,a,b}$, where $n\geq 1$. Then, the norm of the functional $L_n$ is
    \[
    \|L_n\|\leq (|a_n|^q+|b_{n-1}|^q)^{1/q},
    \]
    where $1<p<\infty$ and $1/p+1/q=1$. For $p=1$, we have $\|L_n\|\leq \max\{|a_n|,|b_{n-1}|\}$. Also, for the case of $c_{0,a,b}$, $\|L_n\|\leq |a_n|+|b_{n-1}|$.
\end{proposition}
\begin{proof}
We provide the proof only for the case $1<p<\infty$. If $f \in \ell^{p}_{a,b},$ then
 \[
   f(z)=\sum_{n=0}^{\infty} \lambda_{n}f_{n}(z)=\sum_{n=0}^{\infty}\lambda_{n}(a_{n}+b_{n}z)z^{n},
 \]
 and $\|f\|^p_{\ell^p_{a,b}}=(\sum_{n\geq 0}|\lambda_n|^p)^{1/p}$. Rearranging into a power series, we obtain 
 \[
 f(z)=\lambda_0a_0+\sum_{n\geq 1} (\lambda_na_n+\lambda_{n-1}b_{n-1})z^n, 
 \]
 and so, 
 \[
 L_n(f)=\lambda_na_n+\lambda_{n-1}b_{n-1}.
 \]
 It follows, by H\"{o}lder inequality, that 
 \[
 |L_n(f)|\leq \|f\|_{\ell^p_{a,b}} (|a_n|^q+|b_{n-1}|^q)^{1/q},
 \]
 that is, $\|L_n\|\leq (|a_n|^q+|b_{n-1}|^q)^{1/q}$.
\end{proof}

\section{The weighted shift $B_w$: boundedness and a compact perturbation result}
For a given sequence $w:=\{w_n\}_{n\geq 0}$ of non-zero complex numbers, we define the weighted backward shift $B_w$ on $\ell^p_{a,b}$ and also on $c_{0,a,b}$ to be the operator
\begin{equation}
  (B_wf)(z)=\sum_{n=0}^{\infty}w_{n+1}\lambda_{n+1}z^n,
\end{equation}
for $f (z)=\sum_{n=0}^{\infty}\lambda_{n}z^n$ in $\ell^p_{a,b}$ or $c_{0,a,b}$. We derive conditions for the shift operator $B_w$ to be bounded on these spaces. For this, we express a monomial $z^n$ as an expansion in the orthonormal basis $\{f_n\}$; see \eqref{monomial-expansion} below. Such an expression will help us to find estimates of $\|z^n\|_{\ell^p_{a,b}}$ in terms of $\{a_n\}$ and $\{b_n\}$, (Proposition \ref{estimate}). Indeed, fix $n \geq 0,$ and write $z^{n}=\sum_{j=0}^{\infty}\alpha_{j}f_{j}$ for some $\alpha_{j} \in \mathbb{C}.$ Then $z^{n}=\alpha_{0}a_{0}+\sum_{j=1}^{\infty}(\alpha_{j-1}b_{j-1}+\alpha_{j}a_{j})z^{j}.$ Thus, comparing the coefficients, we have $\alpha_{0}=\alpha_{1}=\cdot \cdot \cdot=\alpha_{n-1}=0,$ and $\alpha_{n}=\frac{1}{a_{n}}$. Since $\alpha_{n+k-1}b_{n+k-1}+\alpha_{n+k}a_{n+k}=0,$ it follows that 
\[
\alpha_{n+k}=-\frac{\alpha_{n+k-1}b_{n+k-1}}{a_{n+k}},
\]
and thus 
\[
\alpha_{n+k}=\frac{(-1)^{k}}{a_{n}}\frac{b_{n}b_{n+1}\cdot\cdot \cdot b_{n+k-1}}{a_{n+1}a_{n+2}\cdot\cdot \cdot a_{n+k}}, \,\,\,\, (k\geq 1).
\]
This implies 
\begin{equation}\label{monomial-expansion}
z^{n}=\frac{1}{a_{n}}\sum_{j=0}^{\infty}(-1)^{j}(\frac{\prod_{k=0}^{j-1}b_{n+k}}{\prod_{k=0}^{j-1}a_{n+k+1}})f_{n+j}, \,\,\,\, (n\geq 0),
\end{equation}
where the term corresponding to $j=0$ is $1$. (Also, see \cite{Adams-McGuire}, p. 727.) The above expansion will be used repeatedly.

 To obtain necessary and sufficient conditions for $B_w$ to be bounded on $\ell^p_{a,b}$ and $c_{0,a,b}$, we compute the matrix of the operator $B_w$ with respect to the normalizes Schauder basis $\{f_n\}_{n\geq 0}$ and then, study the matrix operator on $\ell^p$ and $c_0$. See \cite{Adams-McGuire} for a similar study on tridiagonal shifts.

\begin{proposition} \label{matri-B}
The matrix representation of $B_w$ with respect to the (ordered) Schauder basis $\{f_{n}\}_{n\geq 0}$ is 
 \begin{equation}\label{matrix}
  [B]:=\begin{bmatrix}
 	\frac{w_{1}b_{0}}{a_{0}} & \frac{w_{1}a_{1}}{a_{0}} & 0 & 0 & 0 &\cdots\\
    -\frac{w_{1}b_{0}^{2}}{a_{0}a_{1}} & c_{1} & \frac{w_{2}a_{2}}{a_{1}}  & 0 & 0 &\ddots \\
	\frac{w_{1}b_{0}^{2}b_{1}}{a_{0}a_{1}a_{2}}& -\frac{c_{1}b_{1}}{a_{2}} & c_{2} & \frac{w_{3}a_{3}}{a_{2}} & 0 & \ddots\\
	-\frac{w_{1}b_{0}^{2}b_{1}b_{2}}{a_{0}a_{1}a_{2}a_{3}}& \frac{c_{1}b_{1}b_{2}}{a_{2}a_{3}} & 	-\frac{c_{2}b_{2}}{a_{3}} & c_{3} & 0 &\ddots\\
	\frac{w_{1}b_{0}^{2}b_{1}b_{2}b_{3}}{a_{0}a_{1}a_{2}a_{3}a_{4}}&-\frac{c_{1}b_{1}b_{2}b_{3}}{a_{2}a_{3}a_{4}} & 	\frac{c_{2}b_{2}b_{3}}{a_{3}a_{4}} & -\frac{c_{3}b_{3}}{a_{4}} &\ddots & \ddots\\
	 \vdots & \vdots & \vdots & \ddots&\ddots&\ddots
	\end{bmatrix}.
	\quad
	\end{equation}
\end{proposition}
\begin{proof}
Recall the Schauder basis $f_{n}(z) =(a_{n}+b_{n}z)z^{n},n\geq 0$, in $\ell^p_{a,b}$ and $c_{0,a,b}$. Note that $B_w(f_{n})(z)=w_{n}a_{n}z^{n-1}+w_{n+1}b_{n}z^{n},~n\geq 1$. Also, $f_{0}(z) =a_{0}+b_{0}z$ and
\[
B_w(f_{0})(z)=w_{1}b_{0}=\frac{w_{1}b_{0}}{a_{0}}f_{0}-\frac{w_{1}b_{0}^{2}}{a_{0}a_{1}}f_{1}+\frac{w_{1}b_{0}^{2}b_{1}}{a_{0}a_{1}a_{2}}f_{2}-\frac{w_{1}b_{0}^{2}b_{1}b_{2}}{a_{0}a_{1}a_{2}a_{3}}f_{3}+\cdots.
\]
For $n\geq 1$, we have
\[
B_{w}(f_{n})(z)=w_{n}a_{n}z^{n-1}+w_{n+1}b_{n}z^{n}=\frac{w_{n}a_n}{a_{n-1}}f_{{n-1}} + (\frac{w_{n+1}b_{n}}{a_{n}}-\frac{w_{n}a_n}{a_{n-1}}\frac{b_{n-1}}{a_{n}})a_{n} z^{n },
\]
by putting the value of $z^{n-1}$. Setting
\[
c_{n}:=w_{n+1}\frac{b_{n}}{a_{n}}-w_{n}\frac{b_{n-1}}{a_{n-1}},
\]
we immediately get $
B_{w}(f_{n})(z)=\frac{w_{n}a_{n}}{a_{n-1}}f_{{n-1}}+c_{n}a_{n} z^{n }, ~~~n\geq 1.$ Now, the matrix representation of $B_{w}$ can be obtained from the following expressions:
\[
B_{w}(f_{1})(z)=\frac{w_{1}a_1}{a_{0}}f_{{0}}+c_{1}f_{{1}}-\frac{c_{1}b_{1}}{a_{2}}f_{{2}}+\frac{c_{1}b_{1}b_{2}}{a_{2}a_{3}}f_{{3}}-\cdots,
\]
\[
B_{w}(f_{2})(z)=\frac{w_{2}a_2}{a_{1}}f_{{1}}+c_{2}f_{{2}}-\frac{c_{2}b_{2}}{a_{3}}f_{{3}}+\frac{c_{2}b_{2}b_{3}}{a_{3}a_{4}}f_{{4}}-\cdots,
\]
and so on. These expressions give rise to the matrix of $B_{w}$ with respect to $\{f_n\}_{n=0}^{\infty}$, as in the proposition.
\end{proof}

\noindent Compare the above matrix for the unweighted shift, with that of a left inverse of the multiplication operator $\big(Sf\big)(z)=zf(z)$ defined on a tridiagonal space, cf. Das and Sarkar \cite{Das-Sarkar}, Proposition $3.1$.

We now determine necessary and sufficient conditions under which the above (formal) matrix defines a bounded operator on $\ell^p$ and $c_{0,a,b}$. Equivalently, this gives boundedness results for $B_w$ acting on $\ell^p_{a,b}$ and $c_{0,a,b}$.

\begin{proposition}
If $B_{w}$ is bounded on $\ell^{p}_{a,b}$ and $c_{0,a,b}$, then $\{\frac{w_{n+1}a_{n+1}}{a_{n}}\}_{n\geq 1}$ and $\{c_n\}_{n\geq 1}$ are bounded.
 \end{proposition}
 \begin{proof}
 Let $B_{w}$ be bounded on $\ell^{p}_{a,b}$, $1\leq p<\infty$. Then the matrix $[B]$ induces a bounded operator on $\ell^p$. Let $v_n$ be the $n$-th column of $[B]$. Operating $[B]$ on the subset $\{e_n\}_{n\geq 1}$ of the standard orthonormal basis in $\ell^p$, since $[B](e_n)=v_n$, we get that 
 \[
 \sup_{n}\|v_n\|<\infty.
 \]
 On the other hand,
 \[
 \|v_n\|_{\ell^p}^p\geq \big|\frac{w_{n}a_n}{a_{n-1}}\big|^p+|c_n|^p,\hspace{.5cm} n\geq 1.
 \]
 This implies the necessary conditions for $\ell^p_{a,b}$, as in the proposition. The case of $c_{0,a,b}$ is similar.
 
\end{proof}
   
The following theorem gives a (general) sufficient condition for $B_{w}$ to be bounded. 

 \begin{theorem}\label{decomposition}
   If 
   \[
 \sup_{n\geq 1}~\Big\{\lvert\frac{w_{n+1}a_{n+1}}{a_n}\rvert,~|c_n|\Big\}<\infty,~ \text{and}
 \]
 \[
 \sum_{n=1}^{\infty}\max \left \{ \left|\frac{w_{1}b_0^2b_1\cdots b_{n-1}} {a_0a_1\cdots a_n}\right|, ~\sup_{j\geq 1} ~\left|\frac{c_jb_jb_{j+1}\cdots b_{j+n-1}}{a_{j+1}a_{j+2}\cdots a_{j+n}}\right|\right \}<\infty,
   \]
   then $B_{w}$ is bounded on $\ell^{p}_{a,b}$ and $c_{0,a,b}$, where $1\leq p<\infty$.
   \end{theorem}
   \begin{proof}
Recall the matrix representation of $B_{w}$. We can split this matrix as a formal series of infinite matrices as follows:
 $$
 [B] =
 \begin{bmatrix}
 	0 & \frac{w_{1}a_{1}}{a_{0}} & 0 & 0 & \cdots\\
    0 & 0 & \frac{w_{2}a_{2}}{a_{1}}  & 0 & \ddots \\
	0& 0 &0 &\frac{w_{3}a_{3}}{a_{2}} & \ddots\\
	0& 0 & 	0 & 0  &\ddots\\
	 \vdots & \vdots & \vdots & \ddots & \ddots
	\end{bmatrix}
 +
 \begin{bmatrix}
 	\frac{w_{1}b_{0}}{a_{0}} & 0 & 0 & 0 &\cdots\\
    0 & c_1  & 0 & 0 & \ddots \\
	0 & 0 & c_{2} & 0 & \ddots\\
	0 & 0 & 0 & c_{3} & \ddots\\
	 \vdots & \vdots & \vdots & \ddots&\ddots
	\end{bmatrix}
	 + 
 \begin{bmatrix}
 	0 & 0 & 0 & \cdots\\
    -\frac{w_{1}b_{0}^{2}}{a_{0}a_{1}} & 0 & 0 &\ddots \\
	0 & -\frac{c_{1}b_{1}}{a_{2}} & 0 &\ddots\\
	0 & 0 & 	-\frac{c_{2}b_{2}}{a_{3}} &\ddots\\
	 \vdots & \vdots & \ddots & \ddots
	\end{bmatrix}
	 +\ldots,
 $$
 \vspace{.2cm}
 
\noindent which is a (formal) series  $[B_{\alpha}]+[D]+\sum_{n=1}^{\infty} [F_{n}]$. Here, $[B_{\alpha}]$ is the matrix of the standard weighted backward shift $B_{\alpha}(e_i)\mapsto \alpha_i e_{i-1}$ on $\ell^p$, $i\geq 1$, having weights 
\[
\alpha_i=w_{i}\frac{a_i}{a_{i-1}}, \hspace{1cm} (i\geq 1),
\]
and $[D]$ is the matrix of the diagonal operator 
\[
\text{diag}~ (\frac{w_{1}b_0}{a_0},c_1,c_2,\cdots)
\]
on $\ell^p$. The matrix $[F_n]$ is obtained by deleting all the entries of $[B]$, except those at the $n$-th subdiagonal, where $n\geq 1$. Observe that $[F_n]$ is the matrix of suitable powers of a weighted forward shift $F_n$ for $n\geq 1$.\\ 
It follows, respectively by the first two assumptions in the theorem, that the weighted shift $B_{\alpha}$ and the diagonal operator $D$ are bounded on $\ell^p$. Since
\[
\|F_n\|=\max \left \{ \left|\frac{w_{1}b_0^2b_1\cdots b_{n-1}} {a_0a_1\cdots a_n}\right|, ~\sup_{j\geq 1} ~\left|\frac{c_jb_jb_{j+1}\cdots b_{j+n-1}}{a_{j+1}a_{j+2}\cdots a_{j+n}}\right|\right \},
\]
the third condition in the theorem gives that $F_n$ is bounded, and
\[
\sum_{n\geq 1}\|F_n\|<\infty
\]
with respect to the operator norm. Hence, the shift $B_{w}$ is bounded on $\ell^{p}_{a,b}$. This completes the proof of the theorem.
\end{proof}
\begin{remark}
We note that the unweighted backward shift $B$ is a left-inverse of the multiplication operator $(Sf)(z)= zf(z)$ on $\ell^p_{a,b}$. In the case of $p=2$, a closely related left inverse $B_1$ of $S$ was studied in \cite{Das-Sarkar}, wherein the authors obtained conditions for the boundedness of $B_1$. The matrices of $B$ and $B_1$ are almost the same, except the difference in the first columns. In the weighted case, the following assumptions for the boundedness are strong, compared to those in the above theorem. Indeed, the conditions
\begin{equation}\label{sup-limsup}
\sup_{n\geq 1}~\lvert w_{n+1}\frac{a_{n+1}}{a_{n}}\rvert<\infty \hspace{.4cm} \text{and} \hspace{.4cm} \lim \sup_{n}\left \lvert \frac{b_{n}}{a_{n+1}}\right \rvert<1
\end{equation}
imply those in Theorem \ref{decomposition}. To see this, writing 
\[
c_{n}=\frac{b_{n}}{a_{n+1}}\frac{w_{n+1}a_{n+1}}{a_{n}}-\frac{w_{n}a_n}{a_{n-1}}\frac{b_{n-1}}{a_{n}}, \hspace{.5cm} (n\geq 1),
\]
we can see that $\{c_n\}$ is bounded. Moreover, since $\lim \sup_{n}\left \lvert \frac{b_{n}}{a_{n+1}}\right \rvert<1$, there exist $r<1$ and $N\in \mathbb{N}$ such that $\left \lvert\frac{b_{n}}{a_{n+1}}\right \rvert< r$, for $n\geq N$. From this, the remaining conditions in Theorem \ref{decomposition} follow. 
\end{remark}


Under some mild assumptions on $\{a_n\}$ and $\{b_n\}$, we prove that the shift $B_{w}$ acting on $\ell^p_{a,b}$ and $c_{0,a,b}$ are similar to the sum $B_{\alpha}+K$ on $\ell^p$ and $c_0$, respectively, for a suitable weighted backward shift on $\ell^p$ or $c_0$, and a compact operator $K$. Using this perturbation result, we compute the essential spectrum of the shift $B_{w}$ for the case $1<p<\infty$. These results are of independent interest as well.

An operator $T$ on a Banach space is called \textit{Fredholm} if its kernel Ker $T$ and the quotient $X/T(X)$ of the range of $T$, are of finite dimension. In case $X$ is reflexive, then $T$ is Fredholm if and only if Ker $T$ and Ker $T^*$ are finite dimensional. The essential spectrum $\sigma_e(T)$, of an operator $T$ on a complex Banach space $X$ is the set of all $\lambda\in \mathbb{C}$ such that $T-\lambda I$ is not Fredholm, that is,
\[
\sigma_e(T)=\{\lambda\in\mathbb{C}:~\text{dim~Ker} ~(T-\lambda I)=\infty~\text{or dim Ker} ~(T^*-\lambda I)=\infty\},
\]
where $T^*$ is the adjoint of $T$, cf. Bayart and Matheron \cite{Bayart-Matheron} and Douglas \cite{Douglas}. The essential spectrum plays a key role in the investigation of hypercyclic subspaces; see the section $4$. 

In the proof of the following theorem, we use a well known fact: $\sigma_e(T)$ is invariant under a compact perturbation, that is,
\[
\sigma_e(T+K)=\sigma_e(T)
\]
for every compact operator $K$. 

\begin{theorem}\label{prop}
Assume that 
\[
\sup_n |w_{n+1}\frac{a_{n+1}}{a_{n}}|<\infty, ~\limsup_n |\frac{b_n}{a_{n+1}}|<1,~ and ~\lim_n \left|w_{n+1}\frac{b_n}{a_n}-w_{n}\frac {b_{n-1}}{a_{n-1}}\right|=0. 
\]
Then the following hold.
\begin{enumerate}
\item[(i)] The operator $B_{w}$ on $\ell^p_{a,b}$ (and $c_{0,a,b}$) is similar to $B_{\alpha}+K$ for some compact operator $K$ and the weighted backward shift $B_{\alpha}$ on the sequence space $\ell^p$ ($c_0$, respectively), where the weight sequence $\alpha=(\alpha_n)$ is given by 
\[
\alpha_n=w_{n}\frac{a_{n}}{a_{n-1}},\hspace{0.2cm} n\geq 1.
\]
\item[(ii)] On $\ell^p_{a,b}$, $1<p<\infty$, the essential spectrum $\sigma_e(B_w)$ is the annulus
\[
\sup_{n\geq 1}\left (\inf_{k\geq 1}\left|(w_{k+1}w_{k+2}\cdots w_{k+n})\frac{{a}_{k+n}}{a_k}\right|\right)^{1/n} ~\leq |z|\leq\inf_{n\geq 1} \left(\sup_{k\geq 1}\left|(w_{k+1}w_{k+2}\cdots w_{k+n})\frac{{a}_{k+n}}{a_k}\right|\right)^{1/n}.
\]
\end{enumerate}
\end{theorem}
\begin{proof}
The proof relies on the matrix representation of $B_{w}$ with respect to the orthonormal basis $f_n(z)=(a_n+b_nz)z^n$ of $\ell^p_{a,b}$. Consider the similarity map $U:\ell^p_{a,b}\rightarrow \ell^p$ given by 
\[
U(\sum_{n=0}^{\infty} \lambda_n f_n)=\sum_{n=0}^{\infty}\lambda_ne_n,
\]
that is, $U(f_n)=e_n$ for all $n$, where $\{e_n\}_{n\geq 0}$ is the standard basis in $\ell^p$. Now, from the proof of Theorem \ref{decomposition} we recall that $B_{w}$ on $\ell^p_{a,b}$ is similar via $U$ to the sum (in the operator norm)
\[
B_{\alpha}+D+\sum_{m=1}^{\infty}F_m.
\]
Here, $B_{\alpha}$ in the weighted backward shift on $\ell^p$ with weights
\begin{equation}\label{weights}
\alpha_n=w_{n}\frac{a_{n}}{a_{n-1}}, \hspace{1cm} n\geq 1.
\end{equation}
Further, by the assumptions, the operators $D$ and $F_m$ are compact on $\ell^p$ since the entries in the matrix of $D$ and $F_m$ converge to $0$, $m\geq 1$. Hence, $K:=D+\sum_{m=1}^{\infty}F_m$ is a compact operator on $\ell^p$, and consequently, $B_w$ acting on $\ell^p_{a,b}$ is unitarily equivalent to $B_{\alpha}+K$. This proves $\text{(i)}$.

The invariance of the essential spectrum under compact perturbations along with $\text{(i)}$ yields that
\[
\sigma_e(B_w)=\sigma_e(B_{\alpha}+K)=\sigma_e(B_\alpha).
\]
Thus, it is enough to compute $\sigma_e(B_\alpha)$. We now recall the essential spectrum of a weighted backward shift on $\ell^2$ and refer to \cite{Bayart-Matheron} and \cite{Shields}: In general, for a weighted shift $B_{\alpha}$ corresponding to $\alpha=\{\alpha_n\}_{n=1}^{\infty}$ of non-zero weights, the essential spectrum $\sigma_e(B_{\alpha})$ is the annulus
\[
\sup_{n\geq 1}\left (\inf_{k\geq 1}\prod_{i=1}^{n}|\alpha_{k+i}|\right)^{1/n} ~\leq |z|\leq\inf_{n\geq 1} \left(\sup_{k\geq 1} \prod_{i=1}^n|\alpha_{k+i}|\right)^{1/n}.
\]
The same proof applies to the case of $B_w$ on $\ell^p$, $1<p<\infty$. In our setting, $B_{\alpha}$ is the weighted shift with weights as in \eqref{weights}. Since
\[
\prod_{i=1}^{n}{\alpha}_{k+i}=(w_{k+1}w_{k+2}\cdots w_{k+n})\frac{{a}_{k+n}}{a_k}
\]
for all $k,n\geq 1$, the result in $\text{(ii)}$ follows. The proof is complete.
\end{proof}

\section{Linear dynamics of $B_w$}

In this section, we characterize the hypercyclicity, mixing, and chaos of $B_{w}$ on $\ell^p_{a,b}$ and $c_{0,a,b}$. The following estimates on the norms of monomials will be used in the characterizations of the hypercyclicity properties of $B_{w}$.

\begin{proposition} \label{estimate}
Assume that the conditions in Theorem \ref{decomposition} hold. Then, $\ell^p_{a,b}$ contains all the polynomials densely, $1\leq p<\infty$. Also,
\begin{equation}\label{A}
\|z^n\|_{\ell^p_{a,b}}= \frac{1}{|a_n|} \left (1+\sum_{j=1}^{\infty}\left\lvert \frac{b_{n}b_{n+1}\cdots b_{n+j-1}}{a_{n+1}a_{n+2}\cdots a_{n+j}}\right\rvert^p \right)^{1/p}.
 \end{equation}
In addition, if $\limsup_n|\frac{b_n}{a_{n+1}}|<1$, then there is a constant $M_1>0$ such that 
\begin{equation}\label{B}
\| z^{n}\|_{\ell^{p}_{a,b}} \leq  \frac{M_1}{|a_{n}|},\hspace{.35cm} n\geq 0.
\end{equation}
\end{proposition}

\begin{proof}
Under the hypothesis, the operator $B_w$ is bounded on $\ell^p_{a,b}$. It follows that $w_0b_0=B_w(a_0+b_0z)$ and $w_na_nz^{n-1}+w_{n+1}b_nz^n=B_w(a_nz^n+b_nz^{n+1})$, $n\geq 1$, belong to $\ell^p_{a,b}$. Choosing different values of $n$, we find that the polynomials are contained in $\ell^p_{a,b}$. Indeed, $w_1a_1+w_2b_1z \in \ell^p_{a,b}$ and hence the monomial $z$ (and other powers $z^n$, similarly) belongs to $\ell^p_{a,b}$.

By the basis expansion in $\ell^{p}_{a,b}$ and the continuity of evaluation functionals, we can find some $\{\lambda_{j}\}_{j=0}^{\infty}\in \ell^p$ such that $$z^{n}= \sum_{j \geq 0} \lambda_{j}( a_{j}z^{j}+b_{j}z^{j+1}),$$ for all $z\in \mathbb{D}$. Equating the coefficients of like-powers, we have that $\lambda_{j}=0$ for $j=0,\ldots,n-1$, and
 \begin{equation}
 \lambda_{n}=\frac{1}{a_{n}},~~ \lambda_{n+1}=-\frac{b_{n}}{a_{n+1}}\lambda_{n}=-\frac{1}{a_{n}}\frac{b_{n}}{a_{n+1}},~~ \lambda_{n+2}= \frac{1}{a_{n}} \frac{b_{n}}{a_{n+1}} \frac{b_{n+1}}{a_{n+2}},
 \end{equation}
 and so on. Since $\lVert  z^{n}  \rVert^{p}_{\ell^{p}_{a,b}} = \sum_{j \geq 0} |\lambda_{j}|^{p},$
we have
\begin{equation}\label{O}
 \lVert  z^{n}  \rVert^{p}_{\ell^{p}_{a,b}} = \frac{1}{|a_{n}|^{p}} + \left | \frac{1}{a_{n}}\frac{b_{n}}{a_{n+1}}\right|^p+ \left| \frac{1}{a_{n}} \frac{b_{n}}{a_{n+1}} \frac{b_{n+1}}{a_{n+2}}\right|^p+\cdots,
 \end{equation}
 which gives the norm of $z^n$.
 
 On the other hand, the strong assumption $\limsup_n|\frac{b_n}{a_{n+1}}|<1$ implies that there exist $r<1$ and $N\in \mathbb{N}$ such that $|b_n/a_{n+1}|<r$ for all $n\geq N$. Thus, by the equation \eqref{O} we have $$  \lVert  z^{n}  \rVert^{p}_{\ell^{p}_{a,b}} \leq  \frac{1}{|a_{n}|^{p}} \big (\sum_{k\geq 0 } r^{pk} \big ),$$
 for every $n\geq N$. The required result in the second part of the proposition follows.
 \end{proof}
 The $c_{0,a,b}$ analogue of the above result is as follows, and its proof is omitted.
 
 \begin{proposition}
     Assume that the conditions in Theorem \ref{decomposition} hold. Then 
\begin{equation}\label{D}
 \lVert  z^{n}  \rVert_{c_{0,a,b}} = \max \left \{\frac{1}{|a_n|},\frac{1}{\lvert a_{n}\rvert}\sup_{j\geq 1}\left\lvert \frac{b_{n}b_{n+1}\cdots b_{n+j-1}}{a_{n+1}a_{n+2}\cdots a_{n+j}}\right\rvert \right \}.
 \end{equation}
In addition, if $\limsup_n|\frac{b_n}{a_{n+1}}|<1$, then there is a constant $M_2>0$ such that 
\begin{equation}\label{E}
\| z^{n}\|_{c_{0,a,b}} \leq  \frac{M_2}{|a_{n}|},\hspace{.35cm} n\geq 0.
\end{equation}
 \end{proposition}

We need one more result which also shows the connections between dynamics of the backward shift operator and derivatives of the underlying kernel. 
\begin{proposition}
    Let the weighted backward shift $B_{w}$ be a bounded operator on a Banach space $X$ of analytic functions on the unit disc $\mathbb{D}$, such that $X$ has some functions not vanishing at the origin.
    \begin{enumerate}
        \item[(i)] For some bounded subset $Y$ of $X$, if the orbit $\cup_{n\geq 1}B_{w}^n(Y)$ is dense, then
        \[
        \sup_{n\geq 0} \frac{|w_{1}w_{2}\cdots w_{n}|} {n!} \left\lVert\frac{d^{n}ev_{z}}{dz^{n}}|_{z=0} \right\rVert=\infty.
        \]
        In particular, this is true if $B_w$ is hypercyclic.
        \item[(ii)] If $B_{w}$ is topologically mixing, then
        \[
\lim_{n\rightarrow \infty} \frac{|w_{1}w_{2}\cdots w_{n}|} {n!}\left\lVert\frac{d^{n}ev_{z}}{dz^{n}}|_{z=0} \right\rVert=\infty.
        \]
    \end{enumerate}
\end{proposition}
\begin{proof}
   $\text{(i)}$ Suppose that
   \[
   \{(B_{w})^{n}f:n\geq1, f\in Y\}
   \]
is dense in $X$. It follows from the continuity of the (non-zero) evaluation functional 
   \[
   ev_0(g) = g(0),\hskip 1cm (g\in X).
   \]
   that $\{(B_w^nf)(0):n\geq 1, f\in Y\}$ is dense in $\mathbb{C}$. Now, from the expression
   \[
   (B_w^nf)(0)=w_1w_2\cdots w_n \frac {f^{(n)}(0)}{n!}, ~~n\geq 1,
   \]
   we get that
   \begin{equation}\label{deri}
   \sup_{n\geq 1, f\in Y}~\lvert w_{1}w_{2}\cdots w_{n} \frac{f^{(n)}(0)}{n!}\rvert=\infty.
   \end{equation}
   On the other hand, Proposition \ref{evaluation} shows that the co-ordinate functional defined by
   \[
   \big(\frac{d^n ev_z}{dz^n}|_{z=0}\big)(g)= \frac{g^{(n)}(0)}{n!}, \hskip 1cm g\in X,
   \]
   is bounded. Combining \eqref{deri} with
   \[
   \big |{f^{(n)}(0)}\big |\leq \left\|\frac{d^n ev_z}{dz^n}|_{z=0}\right\| \|f\|_{\ell^p_{a,b}}, ~~~n\geq 1,
   \]
   one gets $\text{(i)}$ as $Y$ is bounded.
  
     To obtain $\text{(ii)}$, we recall from Bonet \cite{Bonet2} that if $T$ is a mixing operator on a Banach space $X$, then 
     \[
     \lim_{n\rightarrow \infty}\|(T^*)^n(x^*)\|=\infty,
     \]
     for each non-zero bounded linear functional $x^*$ on $X$, i.e. 
     \[
     \lim_{n\rightarrow \infty} \|x^*\circ T^n\|=\infty.
     \]
     From this, we find $f\in X$ such that $\lim_{n\rightarrow \infty}|x^*(T^nf)|=\infty.$ In our case, $T=B_{w}$, and take $x^*=ev_0$, and we can see that $\text{(ii)}$ holds, as in $\text{(i)}$.
\end{proof}

\begin{proposition}
Let the weighted backward shift $B_{w}$ be a bounded operator on a (not necessarily tridiagonal) reproducing kernel space $\mathcal{H}(k)$ of analytic functions on the unit disc $\mathbb{D}$, such that $\mathcal{H}(k)$ has some functions not vanishing at the origin.
\begin{enumerate}
    \item[(i)] If $B_{w}$ is hypercyclic, then
    \[
\sup_{n\geq 0} \frac{|w_{1}w_{2}\cdots w_{n}|} {n!}\left(\frac{\partial^{2n}k}{\partial z^n\partial \overline{\zeta}^n} (0,0)\right)^{\frac{1}{2}}=\infty.
\]
\item[(ii)] If $B_{w}$ is topologically mixing, then
\[
\lim_{n\rightarrow \infty} \frac{|w_{1}w_{2}\cdots w_{n}|} {n!}\left(\frac{\partial^{2n}k}{\partial z^n\partial \overline{\zeta}^n} (0,0)\right)^{\frac{1}{2}}=\infty.
\]
\end{enumerate}
\end{proposition}

With the above two general propositions in hand, we can now state and prove the main dynamical results of $B_w$ of this section, using the hypercyclicity and chaoticity criteria.
 
 \begin{theorem}\label{Hypercyclic}
   Assume that the conditions of Theorem \ref{decomposition} hold.
   \begin{enumerate}
       \item[(i)] If 
       \[
       \sup_{n\geq 1} ~\frac{| w_{1}w_{2}\cdots w_{n}a_{n}|}{\left(1+\sum_{j=1}^{\infty}\left\lvert \frac{b_{n}b_{n+1}\cdots b_{n+j-1}}{a_{n+1}a_{n+2}\cdots a_{n+j}}\right\rvert^p \right)^{1/p}} =\infty,
       \]
       then $ B_{w}$ is hypercyclic on $\ell^{p}_{a,b}$, $1\leq p<\infty$.
       \item[(ii)] If 
       \[
       \sup_{n\geq 1}~ \frac{|w_1w_2\cdots w_na_n|}{\max \left\{1, \sup_{j\geq 1}\left\lvert \frac{b_{n}b_{n+1}\cdots b_{n+j-1}}{a_{n+1}a_{n+2}\cdots a_{n+j}}\right\rvert \right\}}
       =\infty,\]
       then $B_w$ is hypercyclic on $c_{0,a,b}$.
       
       \item[(iii)] If $B_{w}$ is hypercyclic on $\ell^p_{a,b}$ or $c_{0,a,b}$, then 
       \[
       \sup_{n\geq 1}~ |w_{1}w_{2}\cdots w_{n}|(|a_n|+|b_{n-1}|)=\infty.
       \]
       \item[(iv)] Now, assume the stronger condition $\limsup_n |b_n/a_{n+1}|<1$. Then, $B_{w}$ is hypercyclic on $\ell^p_{a,b}$ and $c_{0,a,b}$ if and only if $\sup_{n\geq 1} | w_{1}w_{2}\cdots w_{n} a_n|=\infty$. This is further equivalent to the existence of a bounded set $K$ such that the orbit
       \[
       \bigcup_{n=0}^{\infty} B_w^n(K)
       \]
       is dense.
   \end{enumerate}
    \end{theorem}
   
   \begin{proof}
   To get $\text{(i)}$, we recall and apply the Gethner-Shapiro criterion. Let $X_0$ be the space of all polynomials. Then, $X_0$ is dense in $\ell^{p}_{a,b}$ as it contains the normalized Schauder basis $\{(a_n+b_nz)z^n:n\geq 0\}$. Consider the forward weighted shift $S:X_0\rightarrow X_0$  given by
   \[
     S(z^{j})= \frac{z^{j+1}}{w_{j+1}},~\hspace{.0cm} n\geq 0. 
     \]
Trivially, 
$B_{w}Sf=f$, and $(B_{w})^nf\rightarrow 0$, as $~n\rightarrow \infty,
$
for all $f\in X_0$. It suffices to show that, there exists a strictly increasing sequence $\{m_k\}$ of natural numbers such that
\[
 S^{m_k}(z^j)\rightarrow 0,
\]
as $k\rightarrow \infty$, for every monomial $z^j$. Note that
\[
S^n(z^j)=\frac{1}{w_{j+1}w_{j+2}\cdots w_{j+n}} z^{j+n}.
\]

Combining the assumption in $\text{(i)}$ of our theorem, with the first estimate in Proposition \ref{estimate}, we get an increasing sequence $\{d_k\}$ such that $ S^{d_k}(z^n)\rightarrow 0,$ as $k\rightarrow \infty$. Now, Lemma 4.2 of \cite{Erdmann-Peris} completes the proof of $\text{(i)}$.

The proof of $\text{(ii)}$ runs verbatim.
 
 To obtain $\text{(iii)}$, note by the previous proposition that if $B_{w}$ is hypercyclic, then
\[
\sup_n \frac{|w_{1}w_{2}\cdots w_{n}|} {n!}\left\lVert\frac{d^{n}ev_{z}}{dz^{n}}|_{z=0} \right\rVert_{\ell^{p}_{a,b}}=\infty.
\]
Combining this with Proposition \ref{N-L}, it follows that
\[
\sup_n |w_{1}w_{2}\cdots w_{n}| (|a_n|^q+|b_{n-1}|^q)^{1/q}=\infty,
\]
where $q$ is the H\"{o}lder conjugate of $p$ and $1<p<\infty$. Since any two norms are equivalent on $\mathbb{C}^2$, we obtain that $\sup_n |w_{1}w_{2}\cdots w_{n}| (|a_n|+|b_n|)=\infty.$ This completes the proof of $\text{(ii)}$. (The case of $p=1$ is similar.)

To see the part (iv), we proceed as in (i) and (ii) for the sufficiency and necessity, respectively, and along the way, use the second part in Proposition \ref{N-L} and the condition $\limsup_n |b_n/a_{n+1}|<1$. Indeed, if $K$ is a bounded set in $\ell^p$, $1\leq p<\infty$, such that 
\[
\{B_w^nf:~f\in K, n\geq 1\}
\]
is dense, then 
\[
\sup \left \{\left |w_1w_2\cdots w_n \frac{f^{(n)}(0)}{n!}\right|: f\in K, n\geq 1\right\}=\infty.
\]
Since $K$ is bounded and 
\[
|f^{(n)}(0)|\leq \|\frac{d^{n}ev_{z}}{dz^{n}}|_{z=0}\|\leq C(|a_n|+|b_{n-1}),
\]
we get $\sup_n |w_{1}w_{2}\cdots w_{n}| (|a_n|+|b_{n-1}|)=\infty$. Using the boundedness of $\{b_n/a_{n+1}\}$, we see that the condition in $\text{(i)}$ is satisfied, and hence, $B_w$ is hypercyclic.
\end{proof}

Our next results are about the necessary and sufficient conditions for $B_{w}$ to become mixing.
\begin{theorem}
   The following hold for the shift $B_{w}$ on $\ell^{p}_{a,b}$, $1\leq p<\infty$, and $c_{0,a,b}$.
   \begin{enumerate}
        \item If
       \[
       \lim_{n\rightarrow \infty} ~\frac{| w_{1}w_{2}\cdots w_{n}a_{n}|} {\left(1+\sum_{j=1}^{\infty}\left\lvert \frac{b_{n}b_{n+1}\cdots b_{n+j-1}}{a_{n+1}a_{n+2}\cdots a_{n+j}}\right\rvert^p \right)^{1/p}}=\infty,
       \]
       then $ B_{w}$ is mixing on $\ell^{p}_{a,b}$, $1\leq p<\infty$.
       \item[(ii)] If 
       \[
       \lim_{n\rightarrow \infty}~ \frac{|w_1w_2\cdots w_na_n|}{\max \left\{1,~\sup_{j\geq 1}\left\lvert \frac{b_{n}b_{n+1}\cdots b_{n+j-1}}{a_{n+1}a_{n+2}\cdots a_{n+j}}\right\rvert \right\}}
       =\infty,\]
       then $B_w$ is mixing on $c_{0,a,b}$.
       
       \item[(iii)] If $B_{w}$ is mixing on $\ell^p_{a,b}$ or $c_{0,a,b}$, then 
       \[
       \lim_{n\rightarrow \infty}~ |w_{1}w_{2}\cdots w_{n}|(|a_n|+|b_{n-1}|)=\infty.
       \]
       \item[(iii)] Assuming the stronger condition $\limsup_n |b_n/a_{n+1}|<1$, the operator $B_{w}$ is mixing if and only if $\lim_{n\rightarrow \infty} |w_{1}w_{2}\cdots w_{n} a_n|=\infty$.
   \end{enumerate}
    \end{theorem}
    
    \begin{proof}
Proceed exactly as in the proof of the previous theorem by applying the Gethner-Shapiro criterion with $n_k=k$ for $k\geq 1$. 
        \end{proof}

Along the same lines, we obtain a characterization for $B_{w}$  to be chaotic using the chaoticity criterion. For an operator $T$ defined on $X$, a vector $x\in X$ is called a weakly almost periodic vector if there exists an integer $k\geq 1$ such that $T^{kn}(x)$ converges to $x$ weakly, as $n\rightarrow \infty$.

 \begin{theorem}\label{M-T}
 Consider the weighted shift $B_w$ acting on the space $\ell^p_{a,b}$, $1\leq p<\infty$.\\ \emph{(i)} Set
 \[
\lambda_{n,0}=1,\hspace{.6cm} \text{and}\hspace{.6cm} \lambda_{n,j}=(-1)^{j}\frac{b_{n}b_{n+1}\cdot\cdot \cdot b_{n+j-1}}{a_{n+1}a_{n+2}\cdot\cdot \cdot a_{n+j}}, \,\,\,\, (n\geq0, j\geq 1).
\]
If
 \[
 \sum_{n=1}^{\infty} \left|\frac{\lambda_{1,n-1}}{w_{1} a_1}+\cdots+\frac{\lambda_{n,0}}{w_{1}\cdots w_{n}a_n}\right|^p<\infty,
 \]
 then $B_w$ is chaotic on $\ell^p_{a,b}$.\\ \emph{(ii)} If $\sup_n \big|w_{n+1}\frac{a_{n+1}}{a_n}\big|<\infty$ and $\limsup_n \big|b_n/a_{n+1}\big|<1$, then the following are equivalent.
\begin{enumerate}
\item[(i)] $B_{w}$ is chaotic on $\ell^{p}_{a,b}$.
\item[(ii)] $B_{w}$ has a non-trivial periodic vector.
\item[(iii)] $B_w$ has a non-trivial almost periodic vector.
\item[(iv)] $B_w$ has a non-trivial weakly almost periodic vector.
\item[(v)] The series 
\[
\sum_{n=1}^{\infty} \frac{1}{|w_{1}\cdots w_{n}a_{n} \rvert^{p}}
\]
is convergent.
\end{enumerate}
 \end{theorem}

  \begin{proof}
  The arguments in $\text{(i)}\Rightarrow \text{(ii)}\Rightarrow \text{(iii)}\Rightarrow \text{(iv)}$ are trivial.
  
     Suppose that the condition in $\text{(v)}$ holds. We apply the chaoticity criterion to show that $B_w$ is chaotic on $\ell^{p}_{a,b}.$

     \noindent Let $X_0$ be the space of all polynomials. Define $S: X_0 \rightarrow X_0$ given by $S(z^{n})=\frac{1}{w_{n+1}} z^{n+1}$, $n\geq 0$. Clearly, $B_{w}S=I$ on $X_0$, and the series $\sum_{n=0}^{\infty} (B_{w})^{n}(f)$ converges unconditionally for each $f \in X_0.$ It remains to show that the series  $\sum_{n=0}^{\infty} S^{n}(f)$ converges unconditionally, for each $f \in X_0.$ We prove that
     \[
     \sum_{n=1}^{\infty} \frac{1}{w_{1}\cdots w_{n}} z^n
     \]
     is unconditionally convergent in $\ell^{p}_{a,b}$. Recalling the bases expansion from \eqref{monomial-expansion}, for a fixed $n \geq 0,$ we have
\begin{equation}\label{monomial1}
z^{n}=\frac{1}{a_{n}}\sum_{j=0}^{\infty}\lambda_{n,j}f_{n+j}.
\end{equation}
Also,
\begin{align*}
\sum_{n=1}^{\infty} \frac{1}{w_{1}\cdots w_{n}} z^n=&\sum_{n=1}^{\infty} \frac{1}{w_{1}\cdots w_{n}} \left(\frac{1}{a_{n}}\sum_{j=0}^{\infty}\lambda_{n,j}f_{n+j}\right)\\=&\sum_{n=1}^{\infty} \left(\frac{\lambda_{1,n-1}}{w_{1} a_1}+\cdots+\frac{\lambda_{n,0}}{w_{1}\cdots w_{n}a_n}\right)f_n.
\end{align*}

As $\limsup_n |b_n/a_{n+1}|<1$, one gets $N\in\mathbb{N}$ and $r<1$ such that $|b_n/a_{n+1}|<r$ for all $n\geq N$. Hence,
\[
\left|\frac{\lambda_{1,n-1}}{w_{1} a_1}+\cdots+\frac{\lambda_{n,0}}{w_{1}\cdots w_{n}a_n}\right|\leq \frac{r^{n-1}}{|w_{1} a_1|}+\cdots+\frac{1}{|w_{1}\cdots w_{n}a_n|},
\]
for all $n\geq N$. The right hand side of the above inequality is the $n$-th term of an $\ell^1$-convolution of the $\ell^p$ element 
\[
\left\{\frac{1}{w_{1}\cdots w_{n}a_n}\right\}_{n=1}^{\infty},
\]
and hence it is absolutely $p$-summable. Consequently, the series $\sum_{n\geq N} (w_{1}\cdots w_{n})^{-1}z^n$ is convergent. The unconditional convergence occurs because $\{f_n\}$ is an unconditional basis. Hence, $B_{w}$ satisfies the chaoticity criterion, and \text{(i)} follows.

$\text{(iv)}\Rightarrow\text{(v)}$: Let $f(z)=\sum_{n=0}^{\infty}\lambda_{n}f_{n}(z)$ be a non-zero weakly almost periodic vector for $B_{w}$ on $\ell^{p}_{a,b}$, where $f_n(z)=a_nz^n+b_nz^{n+1}$, $n\geq 0$, forms a normalized Schauder basis for $\ell^{p}_{a,b}$. We can certainly take $\lambda_0\neq 0$. Then, for some $m\in \mathbb{N},$ we have
$B_{w}^{km}(f) \to f$ weakly, as $k\to \infty$. It follows, using the continuity of co-ordinate functionals, that
\[
w_{km}w_{km-1}\cdots w_{1}(\lambda_{km-1}b_{km-1}+\lambda_{km}a_{km}) \to \lambda_{0}a_{0} , \hspace{.5cm} \text{as}~~~~k\to \infty.
\]
Since
\[
\frac{1}{\lvert w_{km}w_{km-1}\cdots w_{1} a_{km} \rvert^{p}}=\left\lvert\frac{\lambda_{km-1}\frac{b_{km-1}}{a_{km}}+\lambda_{km}}{w_{km}w_{km-1}\cdots w_{1}(\lambda_{km-1}b_{km-1}+\lambda_{km}a_{km})}\right\rvert^{p}.
\]
Now $\sum_{k=1}^{\infty}\left\lvert \lambda_{km-1}\frac{b_{km-1}}{a_{km}}+\lambda_{km} \right\rvert^{p}<\infty,$ as $\{\lambda_{k}\}\in \ell^{p}$ and $\limsup_k \big|b_k/a_{k+1}\big|<1,$ and also 
\[
\frac{1}{\left\lvert w_{km}w_{km-1}\cdots w_{1}(\lambda_{km-1}b_{km-1}+\lambda_{km}a_{km})\right\rvert^{p}} \to \frac{1}{\left\lvert\lambda_{0}a_{0}  \right\rvert^{p}}, 
\]
as $k\rightarrow \infty$, where $\lambda_{0}a_{0} \neq 0.$ Thus, we get
\[
\sum_{k=1}^{\infty} \frac{1}{\lvert w_{km}w_{km-1}\cdots w_{1} a_{km} \rvert^{p}} < \infty.
\]
Again for $j=1,\cdots,m-1,$ the weak convergence $B_{w}^{km+j}(f) \to B_w^jf$, where $k\to \infty$ implies that
\[
w_{km+j}w_{km+j-1}\cdots w_{1}(\lambda_{km+j-1}b_{km+j-1}+\lambda_{km+j}a_{km+j}) \to w_{j}w_{j-1}\cdots w_{1}(\lambda_{j}a_{j}+\lambda_{j-1}b_{j-1} ),
\] 
as $k\to\infty.$ Once again for  $\{\lambda_{k}\}\in \ell^{p}$ and $\limsup_k \big|b_k/a_{k+1}\big|<1,$ and also $w_{j}\cdots w_{1}(\lambda_{j}a_{j}+\lambda_{j-1}b_{j-1} )\neq 0$, we get
\[
\sum_{k=0}^{\infty}\frac{1}{\left\lvert w_{km+j}w_{km+j-1}\cdots w_{1}a_{km+j} \right\rvert^{p}}< \infty.
\]
 Consequently, the series in $\text{(iii)}$ is convergent.
  \end{proof}

 We now remark on the existence of hypercyclic subspaces for $B_w$ in $\ell^p_{a,b}$. Recall that if the set $HC(T)$ of all hypercyclic vectors of an operator $T$ on a Banach space $X$ contains a closed infinite dimensional subspace (excluding the zero vector), then we say that $T$ has a hypercyclic subspace. It is well known that the essential spectrum of an operator $T$ on a complex Banach space completely characterizes the existence of hypercyclic subspaces, thanks to an important result of Gonz\'{a}lez, Le\'{o}n-Saavedra and Montes-Rodr\'{i}guez \cite{Gonzalez}: if $T$ is a bounded operator satisfying the hypercyclicity criterion in a complex Banach space, then $T$ has a hypercyclic subspace if and only if 
\begin{equation}\label{hc-subspace}
\sigma_e(T)\cap \overline{\mathbb{D}}\neq \phi.
\end{equation}
For details on the study of hypercyclic subspaces and related topics for various classes of operators including the weighted backward shifts, we refer to \cite{Bayart-Matheron}, \cite{Leon}, \cite{Menet}, and \cite{Montes}.

\begin{corollary}
Assume that
\[
\sup_n|w_{n+1}\frac{a_{n+1}}{a_n}|<\infty,~ \limsup_n|b_n/a_{n+1}|<1, ~~and ~~\lim_n |w_{n+1}\frac{b_n}{a_n}-w_{n}\frac {b_{n-1}}{a_{n-1}}|=0.
\]
Then, $B_{w}$ has hypercyclic subspaces in $\ell^p_{a,b}$, $1<p<\infty$, if and only if
\[
\sup_n|w_{1}w_{2}\cdots w_{n} a_n|=\infty \hspace{.7cm} \text{and} \hspace{.7cm} \sup_{n\geq 1}\left (\inf_{k\geq 1}\left|(w_{k+1}w_{k+2}\cdots w_{k+n})\frac{{a}_{k+n}}{a_k}\right|\right)^{1/n}\leq 1.
\]
\end{corollary}
\begin{proof}
The essential spectrum of $B_w$ is an annulus, cf. Theorem \ref{prop}, which intersects the disc $\overline{\mathbb{D}}$ if and only if the inner radius of the annulus is less or equal to $1$. The result follows at once in view of \eqref{hc-subspace} and Theorem \ref{prop}. 
\end{proof}

We illustrate our results with a Hilbert function space $\mathcal{H}$. See \cite{Adams-McGuire} for a similar study about the the unweighted forward shift on a tridiagonal space $\ell^2_{a,b}$. The space $\mathcal{H}$, defined below, is closely related to the Bergman space $A^2$ (the closure of the space of all complex polynomials in the Lebesgue space $L^2$ of the open unit disc in $\mathbb{C}$) such that the operator $B$ on $\mathcal{H}$ is a Hilbert-Schmidt perturbation of the backward shift on $A^2$. Recall that, if $f\in A^2$ has a power series $f(z)=\sum_{n\geq 0} \hat{f}(n)z^n$, then 
\[
\|f\|^2:=\frac{1}{\pi}\iint_{\mathbb{D}}|f(z)|^2dA(z)=\sum_{n\geq 0} \frac{|\hat{f}(n)|^2}{n+1},
\]
where $dA(z)$ is the area measure in $\mathbb{D}$; see, for instance, \cite{Paulsen}. Moreover, $\sqrt{n+1}\hspace{0.04cm} z^n$, $n\geq 0$, forms an orthonormal basis in $A^2$.

Let $\mathcal{H}$ be the Hilbert space, of all analytic functions on the disc, having an orthonormal basis consisting of
\[
f_n(z)=\sqrt{n+1}z^n+z^{n+1}, ~~ n\geq 0.
\]
Then, the evaluation functionals are bounded, and $\mathcal{H}$ is densely and continuously included in the Bergman space $A^2$. Indeed, if $f\in \mathcal{H}$, then there exists some $\{\lambda_n\}\in\ell^2$ such that $f(z)=\sum_{n=0}^{\infty} \lambda_n f_n(z)$ for all $z\in \mathbb{D}$. Rearranging the sum as a power series, we get 
\[
f(z)=\lambda_0a_0+\sum_{n=1}^{\infty} \Big(\lambda_{n-1}+\lambda_n \sqrt{n+1}\Big)z^n,
\]
and so, for some constant $M>0$ we have
\[
\|f\|_{A^2}^2=|\lambda_0|^2+\sum_{n=1}^{\infty} \frac{\big|\lambda_{n-1}+\lambda_n \sqrt{n+1}\big|^2}{n+1}\leq M^2\sum_{n\geq 0}|\lambda_n|^2<\infty,
\]
since $\{\lambda_n\}\in\ell^2$. As $\|f\|_{A^2}\leq M \|f\|_{\mathcal{H}}$, for all $f\in\mathcal{H}$, the kernel space $\mathcal{H}$ is continuously included in $A^2$. Also, by the main results in the previous sections, the shift $B$ is a bounded operator on $\mathcal{H}$, and it is a non-chaotic mixing operator admitting hypercyclic subspaces, as is the case of $B$ on $A^2$. By Theorem \ref{prop}, it also follows that $B$ on $\mathcal{H}$ is unitarily equivalent to a compact perturbation of the Bergman backward shift $B_w$ on $\ell^2$ with weights $w_n=\sqrt{\frac{n+1}{n}}, ~n\geq 1$. We now have a better understanding of $B$ on $\mathcal{H}$, as follows. Recall that an operator $T$ on a Hilbert space is a Hilbert-Schmidt operator if $\|T\|_{HS}^2:=\sum_n\|T(e_n)\|^2<\infty$, for some (and hence for all)  orthonormal basis $\{e_n\}$ of $\mathcal{H}$.

\begin{proposition}
Let $\mathcal{H}$ be the Hilbert space defined as above. Then the inclusion from $\mathcal{H}$ into $A^2$ is continuous and has dense range. Moreover, the shift $B$ on $\mathcal{H}$ is mixing, but not chaotic, and unitarily equivalent to a Hilbert-Schmidt perturbation of the Bergman backward shift.
\end{proposition}
\begin{proof}
    Recalling from the proof of Theorem \ref{decomposition}, we have that $B$ on $\mathcal{H}$ is (unitarily) 
    \[
    B_w+D+\sum_{n\geq 1}F_n, 
    \]
    where $B_w$ is the Bergman backward shift, $D$ is the diagonal operator having the diagonal $(\frac{b_0}{a_0},c_1,c_2,\cdots)$, and $F_n$ is a power of a weighted forward shift. It is now enough to show that $D$ and each $F_n$ are Hilbert-Schmidt, and $\sum_{n\geq 1}F_n$ is convergent in $\|.\|_{HS}$. Since 
    \[
    c_n=\frac{1}{\sqrt{n+1}}-\frac{1}{\sqrt{n}},~n\geq 1,
    \]
     forms a square summable sequence, it follows that $D$ is Hilbert-Schmidt, and we get the same for each $F_n$. Since 
     \[
     \|F_n\|_{HS}\leq \frac {1}{n^2}
     \]
     for large $n$, we immediately get that $\sum_{n\geq 1} F_n$ is absolutely convergent with respect to the Hilbert-Schmidt norm $\|.\|_{HS}$, which completes the proof.
\end{proof}

\subsection{Eigenvectors and dynamics of $\varphi(B_w)$}

For an analytic function $\varphi(z)$, defined on a neighbourhood of the spectrum $\sigma(B_w)$, let $\varphi(B_w)$ denote the operator on $\ell^p_{a,b}$ and $c_{0,a,b}$,  given by the usual holomorphic functional calculus. In this subsection, we identify the eigenvectors for $\varphi(B_w)$ and apply the eigenvalue criteria to deduce the dynamics of $\varphi(B_w)$.

 Let
\[
h_{\mu}(z):=1+\sum_{n=1}^{\infty}\frac{1}{w_{1}\cdots w_{n}}\mu^{n}z^{n},~~z\in\mathbb{D},
\]
where $\mu\in \mathbb{C}$. Then, $h_{\mu}$ is an eigenvector of $B_{w}$, corresponding to the eigenvalue $\mu$, provided $h_{\mu}\in \ell^p_{a,b}$ or $c_{0,a,b}$. Now we will find the conditions for which the eigenvector $h_{\mu}$ belongs to $\ell^{p}_{a,b}.$ or $c_{0,a,b}$.

\begin{proposition}\label{sp}
    The function $h_{\mu}\in \ell^{p}_{a,b}$, $1\leq p<\infty$, if and only if
    \[
    \sum_{n=1}^{\infty} \left\lvert \frac{\lambda_{1,n-1}\mu}{w_{1} a_1}+\cdots+\frac{\lambda_{n,0}\mu^n}{w_{1}\cdots w_{n}a_n}  \right\rvert^{p}<\infty,
    \]
    where $\{\lambda_{n,j}\}$ is as in Theorem \ref{M-T}. Also, $h_{\mu}\in c_{0,a,b}$ if and only if 
    \[
     \left|\frac{\lambda_{1,n-1}\mu}{w_{1} a_1}+\cdots+\frac{\lambda_{n,0}\mu^n}{w_{1}\cdots w_{n}a_n}\right| \to 0, 
    \]
    as $n \to \infty.$
\end{proposition}
\begin{proof}
    By \eqref{monomial-expansion} we have
\begin{align*}
\sum_{n=1}^{\infty} \frac{\mu^n}{w_{1}\cdots w_{n}} z^n=&\sum_{n=1}^{\infty} \frac{\mu^n}{w_{1}\cdots w_{n}} \left(\frac{1}{a_{n}}\sum_{j=0}^{\infty}\lambda_{n,j}f_{n+j}\right)\\=&\sum_{n=1}^{\infty} \left(\frac{\lambda_{1,n-1}\mu}{w_{1} a_1}+\cdots+\frac{\lambda_{n,0}\mu^n}{w_{1}\cdots w_{n}a_n}\right)f_n.
\end{align*}
Here, $\lambda_{n,0}=0$, and 
\[
\lambda_{n,j}=(-1)^{j}\frac{b_{n}b_{n+1}\cdot\cdot \cdot b_{n+j-1}}{a_{n+1}a_{n+2}\cdot\cdot \cdot a_{n+j}}, \,\,\,\, (n\geq 1, j\geq 1).
\]
Here, $f_{n}(z)=(a_{n}+b_{n}z)z^{n}, n\geq 0.$ The given condition in the proposition now ensures that $h_{\mu}\in\ell^p_{a,b}$. Similarly, the result for $c_{0,a,b}$ follows.
\end{proof}

In the following, $\sigma_p(B_w)$ denote the point spectrum of $B_w$.

\begin{theorem}
Suppose that $B_w$ is bounded on $\ell^p_{a,b}$ and $c_{0,a,b}$, $1\leq p<\infty$. If $\varphi(z)$ is a non-constant function, analytic on a neighborhood of the spectrum of $B_w$ and $\varphi(U)\cap \mathbb{T}\neq \phi$ for some open ball $U$ around $0$ and $U\subseteq \sigma_p(B_w)$, then $\varphi(B_w)$ is mixing and chaotic.
\end{theorem}
\begin{proof}
    Note that $B_w(h_{\mu})=\mu h_{\mu}$ and hence 
    \[
    \varphi(B_w)h_{\mu}=\varphi(\mu) h_{\mu},
    \]
    for all $\mu \in U$. In view of Theorem \ref{aa}, it suffices to show the subspaces $X_0$, $X_1$ and $X_2$ are dense in $\ell^p_{a,b}$ and $c_{0,a,b}$, $1\leq p<\infty$, where
    \[
    X_{0}=\textit{span}~\left\{h_{\mu}: \lambda\in U, \lvert \varphi(\mu)\rvert<1\right\},
    \]
    \[
    X_{1}=\textit{span}~\left\{h_{\mu}: 
    \lambda\in U, \lvert \varphi(\mu)\rvert>1\right\},
    \]
    and
    \[
     X_{2}=\textit{span}~\left\{h_{\mu}: \lambda\in U, \varphi(\mu)^n=1, n\in \mathbb{N}\right\}.
    \]
    We prove an auxiliary result that, for a set $\Gamma$ having accumulation points in $\varphi(U)$,
    \[
    span~\{h_{\mu}: \varphi(\mu) \in \Gamma\}
    \]
    is dense in $\ell^p_{a,b}$, $1\leq p<\infty$, from which it follows that $X_0$, $X_1$ and $X_2$ are dense. Note that the set $\varphi^{-1}(\Gamma):=\{\mu\in U: \varphi(\mu)\in \Gamma\}$ has accumulation points in $U$. So, if $x^{*} \in (\ell^{p}_{a,b})^{*}$, and 
    \[
    x^*(h_{\mu})=x^{*}\left( \sum_{n=0}^{\infty}\mu^{n}z^{n}\right)=0,
    \]
    for all $\mu\in \varphi^{-1}(\Gamma)$, then 
    \[
    x^{*}(z^n)=0,~~\forall n\geq 0.
    \]
    As the polynomials are dense in $\ell^p_{a,b}$, $1\leq p
    <\infty$, we conclude that $x^*=0$, the zero functional.

    The above auxiliary result along with the eigenvalue criteria yields that $\varphi(B_w)$ is mixing and chaotic.
\end{proof}

\begin{example}
    Take $a_0=b_0=1$, and $a_{n}=n^2$ and $b_{n}=1$ for all $n\geq 1$. Consider the (unweighted) shift $B$ on and $\ell^2_{a,b}$. Then, $B$ is a bounded operator by Theorem \ref{decomposition}. Additionally, it fulfils the requirements of hypercyclicity, mixing, and chaos in $\ell^2_{a,b}$. Now, for $|\mu|\leq 1$, we have
    \begin{align*}
    x_n:&= \left\lvert \frac{\mu^n}{a_{n}}-\frac{\mu^{n-1}}{a_{n-1}}\frac{b_{n-1}}{a_{n}}+\cdots +\mu {(-1)^{n-1}} \frac{1}{a_{1}}\frac{b_{1}b_{2}\cdots b_{n-1}}{a_{1}a_{2}\cdots a_{n}}  \right\rvert\\&=\left\lvert \frac{\mu^n}{n^2}-\frac{\mu^{n-1}}{(n-1)^{2}}\frac{1}{n^{2}}+\cdots +\mu{(-1)^{n-1}}\frac{1}{n!^2}\right|  \leq \frac {1}{n}.
    \end{align*}
    In view of the above proposition, $h_{\mu}$ is an eigenvector for $B_w$, corresponding to the eigenvalue $\mu$.    
\end{example}

\begin{remark}
Godefroy and Shapiro \cite{Godefroy-Shapiro} initially studied the dynamics of the differentiation operator $D(f)=f^{\prime}$ acting on the space $\mathcal{H}(\mathbb{C})$ of all entire functions. Indeed, they proved that any non-trivial operator $T$ such that $TD=DT$ is chaotic. Note that $D$ on $\mathcal{H}(\mathbb{C})$ is a weighted backward shift $B_w$ with respect to the basis $\{z^n:n\geq 0\}$, where $w_n=n$, $n\geq 1$.  Motivated by these work, we can understand the dynamics of $D$ on $\ell^p_{a,b}$ and $c_{0,a,b}$, which follow directly from the previous sections by taking $w_n=n$, $n\geq 1$. In particular, the boundedness, hypercyclicity and chaos of $D$ on $\ell^p_{a,b}$ and $c_{0,a,b}$ follow from the results in the sections $3$ and $4$. For some results on the dynamics of differentiation operators, we refer to \cite{Bonet1}, \cite{Chan-Shapiro}, \cite{A-F} and \cite{F}.
\end{remark}

\begin{remark}
It is well known that if an operator $T$ satisfies the chaoticity criterion, then it is frequently hypercyclic, see \cite{Bonilla-Erdmann1}. (The notion of frequent hypercyclicity was introduced by Bayart and Grivaux \cite{Bayart-Grivaux}). Hence, any of the statements in Theorem $4.7$ implies that $B_w$ is frequently hypercyclic on $\ell^p_{a,b}$. It would be interesting to know if it is also necessary for $B_w$ to be frequently hypercyclic. For weighted backward shifts on $\ell^p$, $1\leq p<\infty$, it is well known that the chaos and frequent hypercyclicity are equivalent, cf. \cite{Bayart-Ruzsa}, but not on $c_0$.
\end{remark}

{ \bf Author Contributions.} All authors contributed equally for the preparation of the manuscript.
\vskip .3cm

{\bf Funding.} The first author acknowledges a research fellowship of CSIR,\\ File No.: 09/1059(0037)/2020-EMR-I, and the second named author acknowledges SERB, DST, File. No.: SRG/2021/002418, for a start-up grant.
\vskip .3cm
{\bf Declarations.} The authors have no conflicts of interests.

\bibliographystyle{amsplain}

\end{document}